\documentclass{amsart}
\usepackage{amsmath}
\usepackage{amssymb}

\makeatletter
\@addtoreset{equation}{section}
\makeatother

\newtheorem{theorem}{Theorem}[section]
\newtheorem{proposition}[theorem]{Proposition}
\newtheorem{lemma}[theorem]{Lemma}
\newtheorem{corollary}[theorem]{Corollary}

\theoremstyle{definition}

\newcommand{\what}[1]{\widehat{#1}}

\newcommand{\norm}[1]{\left\|#1\right\|}
\newcommand{\inprod}[2]{\left.\left\langle#1\right|#2\right\rangle}

\newcommand{\spn}{\operatorname{span}}

\newcommand{\fA}{\mathcal{A}}
\newcommand{\fB}{\mathcal{B}}
\newcommand{\fC}{\mathcal{C}}
\newcommand{\fH}{\mathcal{H}}

\newcommand{\fK}{\mathcal{K}}
\newcommand{\fV}{\mathcal{V}}
\newcommand{\fX}{\mathcal{X}}
\newcommand{\fY}{\mathcal{Y}}

\newcommand{\Cee}{\mathbb{C}}

\newcommand{\En}{\mathbb{N}}

\newcommand{\Del}{\Delta}

\newcommand{\gam}{\gamma}
\newcommand{\Gam}{\Gamma}
\newcommand{\lam}{\lambda}
\newcommand{\Lam}{\Lambda}
\newcommand{\sig}{\sigma}

\newcommand{\mat}{\mathrm{M}}
\newcommand{\fal}{\mathrm{A}}
\newcommand{\falg}{\mathrm{A}(G)}
\newcommand{\fdelg}{\mathrm{A}_\Delta(G)}
\newcommand{\fgamg}{\mathrm{A}_\gamma(G)}
\newcommand{\matr}{\mathrm{M}}
\newcommand{\smat}{\mathrm{S}}

\newcommand{\tr}{\mathrm{Tr}}
\newcommand{\trans}{\mathrm{t}}
\newcommand{\trig}{\mathrm{Trig}}
\newcommand{\vn}{\mathrm{VN}}
\newcommand{\cent}{\mathrm{Z}}

\begin{document}

\title[Convolutions on  Fourier algebras]
{Convolutions on the Haagerup tensor products of Fourier algebras}
\author[M. Rostami and N. Spronk]
{Mehdi Rostami and Nico Spronk}

\begin{abstract}
We study the ranges of the maps of convolution $u\otimes v\mapsto u\ast v$ 
and a `twisted' convolution $u\otimes v\mapsto u\ast \check{v}$ ($\check{u}(s)=u(s^{-1})$)
and on the Haagerup tensor product of a Fourier algebra
of a compact group $\falg$ with itself.   We compare the results to result of factoring these
maps through projective and operator projective tensor products.  We notice
that $(\fal(G),\ast)$ is an operator algebra and observe an unexpected set of spectral synthesis.
\end{abstract}
 
\maketitle
 
\footnote{{\it Date}: \today.


M. Rostami.
Faculty of Mathematical and Computer Science, Amirkabir University of Technology, 
424 Hafez Avenue, 15914 Tehran, Iran. {\tt mross@aut.ac.ir}

N. Spronk.
Department of Pure Mathematics, University of Waterloo, 200 University Avenue West,
Waterloo, Ontario, N2L 3G1, Canada. {\tt nspronk@uwaterloo.ca}

M. Rostami  conducted this work whilst visiting University of Waterloo.  
N. Spronk was partially supported by NSERC Grant 312515-2010.}

\section{Introduction}

Let $G$ be a compact group and $\fal(G)$ be its Fourier algebra, in the sense of \cite{eymard}.

In \cite{forrestss}, questions of the following nature were addressed:  what are the ranges
of convolution and `twisted' convolution, when applied to $\fal(G\times G)=\fal(G)\hat{\otimes}\fal(G)$
(operator projective tensor product).  The authors' motivation was two-fold.  First, these 
particular maps played a fundamental role in the famous result of B.~Johnson (\cite{johnson})
that $\fal(G)$ is sometimes non-amenable, and the authors were interested in seeing how these
techniques related to the completely bounded theory of Fourier algebras.  This perspective
led the authors to the results of \cite{forrestss2}. 
Secondly, it was observed is that `twisted' convolution
averages $\fal(G\times G)$ over left cosets of the diagonal subgroup $\Del=\{(s,s):s\in G\}$,
whereas convolution averages $\fal(G\times G)$ over orbits of the group action
$(r,(s,t))\mapsto (sr^{-1},rt):G\times(G\times G)\to G\times G$.  Thus the images may be
rightly regarded as Fourier algebras of certain homogeneous/orbit spaces of 
$G\times G$ in the sense
of \cite{forrest}.  The homogeneous space $G\times G/\Del$ and the orbit space
$G\times G/\check{\Del}$ may be naturally identified with $G$.  Thus we define
$\Gam,\check{\Gam}:\fC(G\times G)\to\fC(G)$ by
\[
\Gam u(s)=\int_Gu(sr,r)\,dr\text{ and }\check{\Gam}u(s)=\int_Gu(sr,r^{-1})\,dr.
\]
It is easy to check that $\Gam(u\otimes v)=u\ast\check{v}$ and $\check{\Gam}(u\otimes v)
=u\ast v$.

In \cite{effrosr1}, it was shown that the Haagerup tensor product
$\fal(G)\otimes^h\fal(G)$ is a Banach algebra.
By \cite[Thm.\ 2]{tomiyama} this algebra has spectrum $G\times G$.
We shall note, below, that $\fal(G)\otimes^h\fal(G)$ is, in fact, semi-simple, and may thus be 
regarded as an algebra of functions on $G\times G$.  Hence it is natural to ask whether
we discover anything new if we apply the maps $\Gam$ and $\check{\Gam}$ to 
$\fal(G)\otimes^h\fal(G)$.  While we gain no new spaces, we learn interesting
comparisons between $\fal(G)\otimes^h\fal(G)$, $\fal(G\times G)$ and 
$\fC(G)\otimes^h\fC(G)$.  See \S \ref{sec:comparison}, below.

\subsection{Some basic results}
We let for each $\pi$ in $\what{G}$, 
$\trig_\pi=\spn\{s\mapsto\inprod{\pi(s)\xi}{\eta}:\xi,\eta\in\fH_\pi\}$ and
$\trig(G)=\bigoplus_{\pi\in\what{G}}\trig_\pi$.  This has linear dual space
$\trig(G)'=\prod_{\pi\in\what{G}}\fB(\fH_\pi)$ via dual pairing
\begin{equation}\label{eq:trigdual}
\langle u, T\rangle=\sum_{\pi\in\what{G}}d_\pi\tr(\hat{u}(\pi)T_\pi)
\end{equation}
where $\hat{u}(\pi)=\int_Gu(s)\pi(s^{-1})\,ds$.  
If $T\in\trig(G)'$, we let $\check{T}$ be defined by
$\langle u,\check{T}\rangle=\langle \check{u},T\rangle$ in the duality (\ref{eq:trigdual}).
Here $\check{u}(s)=u(s^{-1})$.
The following is surely well-known, and recorded here for later convenience.

\begin{lemma}\label{lem:checkdual}
For $T\in\trig(G)'$, then for $\pi$ in $\what{G}$, $\check{T}_\pi={T_{\bar{\pi}}}^\trans$.
Here the conjugation and transpose are realised with respect to the same orthonormal basis
for $\fH_\pi$.
\end{lemma}

\begin{proof}
We first compute that
\[
\what{\check{u}}(\pi)=\int_G u(s)\pi(s)\,ds=\int_G u(s)\bar{\pi}(s^{-1})^\trans\,ds=\hat{u}(\bar{\pi})^\trans.
\]
Hence $\tr(\what{\check{u}}(\pi)T_\pi)=\tr(\hat{u}(\bar{\pi}){T_\pi}^\trans)$.  Thus in
(\ref{eq:trigdual}) we simply change $\pi$ to $\bar{\pi}$.
\end{proof}

We will identify the left regular representation up to quasi-equivalence as
\[
\lam=\bigoplus_{\pi\in\what{G}}\pi\quad\text{on}\quad\fH=\ell^2\text{-}\bigoplus_{\pi\in\what{G}}\fH_\pi.
\]
It is standard that $\lam(G)$ is weak*-dense in $\trig(G)'$ in terms of the duality
(\ref{eq:trigdual}).  The von Neumann algebra generated by $\lam(G)$ is thus the operator space
direct product
\[
\vn(G)=\ell^\infty\text{-}\bigoplus_{\pi\in\what{G}}\fB(\fH_\pi).
\]
Observe that this algebra has centre $\cent\vn(G)
=\ell^\infty\text{-}\bigoplus_{\pi\in\what{G}}\Cee I_\pi$.
The following is also surely well-known.  Again, we provide a proof for convenience of the reader.

\begin{proposition}\label{prop:expectation}
For each $\pi$ in $\what{G}$, we have that $A\mapsto\int_G\pi(s^{-1})A\pi(s)\,ds:
\fB(\fH_\pi)\to\Cee I_\pi$ is a tracial expectation; hence
$\int_G\pi(s^{-1})A\pi(s)\;ds=\frac{1}{d_\pi}\tr(A)I_\pi$.  Thus
$T\mapsto \int_G\lam(s^{-1})T\lam(s)\,ds:\vn(G)\to\cent\vn(G)$ is a tracial expectation
given by 
\[
\int_G\lam(s^{-1})T\lam(s)\;ds=\bigoplus_{\pi\in\what{G}}\frac{1}{d_\pi}\tr(T_\pi)I_\pi.
\]
\end{proposition}

\begin{proof}
It is easy to see that $\int_G\pi(s^{-1})A\pi(s)\,ds$ commutes with each element of $\pi(G)$,
hence by Schur's lemma is a scalar operator, and preserves $I_\pi$.  
Hence, $\int_G\pi(s^{-1})A\pi(s)\;ds=\omega(A)I_\pi$
for some functional $\omega$.  Likewise
\begin{align*}
\int_G\pi(s^{-1})A\pi(t)\pi(s)\,ds&=\int_G\pi(s^{-1})A\pi(ts)\,ds \\
&=\int_G\pi(s^{-1}t)A\pi(s)\,ds=\int_G\pi(s^{-1})\pi(t)A\pi(s)\,ds
\end{align*}
and as $\spn\pi(G)=\fB(\fH_\pi)$ we see that 
$\int_G\pi(s^{-1})AB\pi(s)\,ds=\int_G\pi(s^{-1})BA\pi(s)\,ds$.  By uniqueness of the
normalised trace, $\omega=\frac{1}{d_\pi}\tr$.  The second result is immediate. \end{proof}

The Fourier algebra is the predual of $\vn(G)$ via the dual pairing (\ref{eq:trigdual}).
Hence we obtain complete isometric identification
\begin{equation}\label{eq:falg}
\fal(G)=\ell^1\text{-}\bigoplus_{\pi\in\what{G}}d_\pi\smat^1_{d_\pi}
\end{equation}
where $\smat_d^1$ denotes the $d\times d$-matrices with trace norm; i.e.\ for
$u$ in $\fal(G)$ we have
\[
\norm{u}_{\fal}=\sum_{\pi\in\what{G}}d_\pi\norm{\hat{u}(\pi)}_{\smat^1}
\]
where $\norm{\hat{u}(\pi)}_{\smat^1}$ is the trace-norm of the $d_\pi\times d_\pi$-matrix
$\hat{u}(\pi)$.  

An {\it operator space structure} on a given complex linear space $\fX$ is an assignment
of norms, one on each space $\mat_n(\fX)$ for natural number $n$, which satisfies
certain compatibility conditions; see, for example M1 and M2 of \cite[p.\ 20]{effrosr}.
We shall not require these explicit axioms here.  Of importance to us, are the following facts. First, any
von Neuman algebra $\fV$, in particular $\vn(G)$, will have assigned to each $\mat_n(\fV)$ the
unique norm which realises it as a von Neumann algebra.  Maps
between operator spaces, $\Phi:\fX\to\fY$ are those maps whose matrix amplifications 
$\Phi^{(n)}:\mat_n(\fX)\to\mat_n(\fY)$, $\Phi^{(n)}[X_{ij}]=[\Phi(X_{ij})]$ are
uniformly bounded:  $\norm{\Phi}_{cb}=\sup_n\norm{\Phi^{(n)}}<\infty$.  The space
$\fC\fB(\fX,\fY)$ of completely bounded maps is itself an operator space via the 
isometric identifications $[\Phi_{ij}]\mapsto (X\mapsto[\Phi_{ij}(X)]):\mat_n(\fC\fB(\fX,\fY))\to
\fC\fB(\fX,\mat_n(\fY))$, where $\mat_m(\mat_n(\fX))=\mat_{mn}(\fX)$.  In particular
linear functionals are automatically completley bounded, and $\fX^*=\fC\fB(\fX,\Cee)$
inherits the operator space structure perscribed above.  If $\fV$ is a von Neumann algebra,
then its predual $\fV_*$ inherits the operator space structure from the inclusion $\fV_*\hookrightarrow
\fV^*$.  A map $\Phi:\fX\to\fY$ is called a complete isometry of each $\Phi^{(n)}$ is an isometry, and
a complete quotient if each $\Phi^{(n)}$ is a quotient map.  In the latter case we say $\fY$ is
a complete quotient space of $\fX$.  See \cite[I.3]{effrosr} for details.

Now if $\fA$ is any Banach space
of functions on $G$ for which $\trig(G)$ is dense in $\fA$, then $\fA^*$ may be realised
as a subspace of $\trig(G)'$.  Generally $\fA$ will be a subspace of $\fal(G)$ 
(these will arise form application of $\Gam$ and $\check{\Gam}$; see
\cite[Prop.\ 1.3]{forrestss}), whence
$\fV=\vn(G)$ will be frequently used as below.

Given an operator space $\fA$, a subspace $\fV$ of $\fA^*$ is {\it completely norming} 
if for each $n$ and $[x_{ij}]$ in $\matr_n(\fA)$ we have
\[
\norm{[x_{ij}]}_{\matr_n(\fA)}=
\sup\left\{\norm{[f_{kl}(x_{ij})]}_{\matr_{mn}}:[f_{kl}]\in\matr_m(\fV),
\norm{[f_{kl}]}_{\matr_m(\fA^*)}\leq 1,m\in\En\right\}.
\]
If $\fV$ is any weak*-dense subspace of $\fA^*$, then it is completely norming.
Indeed, this follows from the fact that the embedding of $\fA$ into $\fA^{**}$ is 
a complete isometry \cite[3.2.1]{effrosr}, and then that $\matr_n(\fV)$ is weak*-dense in
$\matr_n(\fA^*)$, where the latter is identified as $\fC\fB(\fA,\matr_n)
\cong(\fA\what{\otimes}\smat^1_n)^*$ \cite[(7.1.6)]{effrosr}.

\begin{lemma}\label{lem:opspace}
Let $\fX$ and $\fY$ be  operator spaces, $\Lam:\fX\to\fY$ be a bounded linear map
with dense range, and $\fV$ be a weak*-dense subspace of $\fY^*$.  Then
$\fA=\Lam(\fX)$, assigned the operator space structure in such a manner that it is regarded as a
complete quotient of $\fX$, admits $\fV$ as a weak*-dense subspace
of it dual $\fA^*$.  In particular, when $\fV$ is given the operator space structure by which 
$\Lam^*:\fV\to\fX^*$ is a complete isometry, then $\fV$ is completley norming for $\fA$.
\end{lemma}

\begin{proof}
The density of $\fA$ in $\fY$ allows us to view $\fV$ as a subspace of $\fA^*$.
Furthermore, if $a$ in $\fA$ satisfies $f(a)=0$ for each $f$ in $\fV$, then $f(a)=0$ for
each $f$ in $\fY^*$, hence $a=0$.  Thus $\{f\in\fV:f|_\fA=0\}=\{0\}$,
and bipolar theorem allows us to conclude that $\fV$ is weak*-dense in $\fY^*$.
Moreover, $\Lam^*$ is a complete isometry exactly when $\Lam$ is a complete quotient map,
thanks to \cite[4.1.9]{effrosr}.  Hence we see that $\fV$ is completely norming, thanks
to remarks in the last paragraph, above.
\end{proof}

\subsection{The Haagerup tensor product of Fourier algebras}
Fix a Hilbert space $\fH$.  In \cite{smith,blechers} it is shown that
each weak*-weak* continuous completley bounded operator $\Phi$ on $\fB(\fH)$ --- we shall 
write $\Phi\in\fC\fB^\sig(\fB(\fH))$ --- is given by
\begin{equation}\label{eq:ncb}
\Phi(T)= \sum_{i\in I}V_iTW_i
\end{equation}
where $\{V_i,W_i\}_{i\in I}$ is a family in $\fB(\fH)$ for which each of the series 
$\sum_{i\in I}V_iV_i^*$ and
$\sum_{i\in I}W_i^*W_i$ is  weak*-convergent.  We shall write $\Phi=
\sum_{i\in I}V_i\otimes W_i$, accordingly.  Furthermore,
we have completely bounded norm
\[
\norm{\Phi}_{cb}=\min\left\{\norm{\sum_{i\in I}V_iV_i^*}^{1/2}
\norm{\sum_{i\in I}W_i^*W_i}^{1/2}:\Phi=\sum_{i\in I}V_i\otimes W_i\text{ as in (\ref{eq:ncb})}\right\}
\]
and operator composition
\begin{equation}\label{eq:cbcomposition}
\sum_{i\in I}V_i\otimes W_i\circ\sum_{i'\in I}V'_{i'}\otimes W'_{i'}
=\sum_{i\in I}\sum_{i'\in I}V_iV'_{i'}\otimes W'_{i'}W_i.
\end{equation}
Hence it is sensible to write
\[
\fC\fB^\sig(\fB(\fH))=\fB(\fH)\otimes^{w^*h}\fB(\fH)
\]
and call this space the {\it weak* Haagerup tensor product}.  This is known to be the
same as the extended Haagerup tensor product of \cite{effrosr1}; see,
for example, the treatment of \cite[\S 2]{spronk}.

Let $\fV\subseteq\fB(\fH)$ be a von Neumann algebra and $\fV'$ its commutant.  It is
shown in \cite{smith,blechers} that subspace of $\fV'$-bimodule maps in 
$\fC\fB^\sig(\fB(\fH))$ are exactly those element which admit a representation as in 
(\ref{eq:ncb}) with each $V_i,W_i$ in $\fV$.  The description of the norm, with minimum
taken over elements $V_i$ and $W_i$ form $\fV$, and the operator composition are maintained,
making this a closed subalgebra of $\fB(\fH)\otimes^{w^*h}\fB(\fH)$.  We denote this space by 
$\fV\otimes^{w^*h}\fV$.  Let $\fV_*$ denote the predual of $\fV$.  We define for an 
elementary tensor $u=v\otimes w$ in $\fV_*\otimes\fV_*$ and $\Phi=\sum_{i\in I}V_i\otimes W_i$
in $\fV\otimes^{w^*h}\fV$, the dual pairing
\[
\langle u,\Phi\rangle
=\sum_{i\in I}\langle v,V_i\rangle\langle w,W_i\rangle.  
\]
and define the Haagerup norm on $\fV_*\otimes\fV_*$ by
\[
\norm{u}_h=\sup\{|\langle u,\Phi\rangle|:\Phi\in\fV\otimes^{w^*h}\fV,\norm{\Phi}_{cb}\leq 1\}.
\]
We then let $\fV_*\otimes^h\fV_*$
denote the completion of $\fV_*\otimes\fV_*$ with respect to this norm.
Then, as shown in \cite{blechers} the dual pairing above extends to a duality 
\begin{equation}\label{eq:htpdual}
(\fV_*\otimes^h\fV_*)^*\cong\fV\otimes^{w^*h}\fV.   
\end{equation}
Then $\fV_*\otimes^h\fV_*$
gains it operator space structure it gains as being a distinguished predual of
$\fV\otimes^{w^*h}\fV\subset\fC\fB^\sig(\fB(H))$, as in \cite[(3.2.2) \& (3.2.11)]{effrosr}.
[The more traditional method
of defining $\fV_*\otimes^h\fV_*$ is to assign $\fV_*$ the predual operator space structure
and use tensor formulas such as in \cite[II.9]{effrosr}.  This gives an equivalent but non-intuitive
description of $\fV_*\otimes^h\fV_*$.  We will make extensive use only of (\ref{eq:htpdual}), above.]

Let us return to the Fourier algebra $\fal(G)$ of a compact group.  
Recall, as in the last section, that $\fal(G)$ is the predual of
the von Neumann algebra $\vn(G)\subset\fB(\fH)$ where 
$\fH=\ell^2\text{-}\bigoplus_{\pi\in\what{G}}\fH_\pi$.
Then we define the Haagerup tensor product of $\fal(G)$ with itself in terms of the completely
isometric duality, as in (\ref{eq:htpdual}), above
\begin{equation}\label{eq:fahtpdual}
(\fal(G)\otimes^h\fal(G))^*\cong\vn(G)\otimes^{w^*h}\vn(G)\subset \fC\fB(\fB(\fH)).
\end{equation}

Let us note a result promised above.

\begin{proposition}
The Banach algebra $\fal(G)\otimes^h\fal(G)$ is semi-simple.
\end{proposition}

\begin{proof}
We observe that $\vn(G)\otimes\vn(G)$ is weak*-dense in $\vn(G)\otimes^{w^*h}\vn(G)\cong
(\fal(G)\otimes^h\fal(G))^*$, by \cite[Thms.\ 3.1 \& 3.2]{blechers}.  It is obvious that
$\spn\lam(G)\otimes\lam(G)$ is weak*-dense in $\vn(G)\otimes\vn(G)$,
where $\lam:G\to\vn(G)$ identifies $G$ with the spectrum of $\falg$.
Thus the bipolar theorem tells us that $(\spn\lam(G)\otimes\lam(G))_\perp=\{0\}$,
whence $\lam(G)\otimes\lam(G)$ separates points in $\fal(G)\otimes^h\fal(G)$.
\end{proof}

\section{Main results}

\subsection{`Twisted' convolution}\label{sec:twistedconvolution}
Our main method for dealing with understanding $\Gam$, as defined on either
of $\fal(G)\hat{\otimes}\fal(G)$, or on $\fal(G)\otimes^h\fal(G)$, is to study its
adjoint.  To this end consider $\Gam:\trig(G)\otimes\trig(G)\to\trig(G)$.  If
$u,v\in\trig(G)$ and $t\in G$ we have
\[
\langle \Gam(u\otimes v),\lam(t)\rangle=
\int_G u(s)v(t^{-1}s)\; ds=\int_G\langle u\otimes v,\lam(s)\otimes\lam(t^{-1})\lam(s)\rangle\,ds
\]
and hence we have that $\Gam^*(\lam(t))=\int_G\lam(s)\otimes\lam(t^{-1})\lam(s)\, ds$, in a 
weak* sense.  By weak*-density of $\spn\lam(G)$ in $\trig(G)'$ we conclude that
for $T$ in $\trig(G)'$ we have
\begin{equation}\label{eq:gamadjoint}
\Gam^*(T)=\int_G\lam(s)\otimes\check{T}\lam(s)\, ds
\end{equation}
where the integral is understood in the weak* sense and $\check{T}$ is defined as in
Lemma \ref{lem:checkdual}.

We define $\fal_\Del(G)=\Gam(\fal(G)\hat{\otimes}\fal(G))$.  In \cite{forrestss} this was
regarded as a quotient space of $\fal(G\times G)\cong\fal(G)\hat{\otimes}\fal(G)$ and
assigned a norm accordingly.
We augment the concrete realisation of this norm, computed in
\cite[Thm.\ 2.2]{forrestss}, by specifying the operator space structure on 
$\fdelg$ in a concrete manner.

\begin{theorem}\label{theo:opsponfdel}
The operator space structure on $\fdelg$, qua complete quotient of $\fal(G)\hat{\otimes}\fal(G)$
by $\Gam$, is given by the weighted operator space direct sum
\[
\fdelg=\ell^1\text{-}\bigoplus_{\pi\in\what{G}}d_\pi^{3/2}\smat^2_{d_\pi,r}
\]
where $\smat^2_{d,r}$ denotes the $d\times d$ matrices with
Hilbert-Schmidt norm and row operator space structure.
\end{theorem}

\begin{proof}
We recall that $(\fal(G)\hat{\otimes}\fal(G))^*\cong\vn(G)\bar{\otimes}\vn(G)$
by \cite{effrosr0} (see \cite[7.2.4]{effrosr}).
We consider the adjoint $\Gam^*:\vn(G)\to\vn(G)\bar{\otimes}\vn(G)$ with
$n$th amplification $\Gam^{*(n)}:\mat_n(\vn(G))\to\vn(G)\bar{\otimes}\mat_n(\vn(G))$.
We have from (\ref{eq:gamadjoint}) that 
\[
\Gam^{*(n)}[T_{ij}]=\int_G \lam(t)\otimes[\check{T}_{ij}] \lam(t)^{(n)}\,ds 
\]
and hence
\begin{align*}
\Gam^{*(n)}[T_{ij}]^*&\Gam^{*(n)}[T_{ij}]
=\int_G\int_G \lam(s^{-1}t)\otimes \lam(s^{-1})^{(n)}[\check{T}_{ij}]^*[\check{T}_{ij}]
\lam(t)^{(n)}\,ds\,dt \\
&=\int_G\int_G \lam(s^{-1})\otimes \lam(s^{-1})^{(n)}\lam(t^{-1})^{(n)}
\sum_{k=1}^n[\check{T}_{ki}^*\check{T}_{kj}]\lam(t)^{(n)}\,ds\,dt \\
&=\int_G\lam(s)\otimes \lam(s)^{(n)}\,ds\cdot
\sum_{k=1}^nI\otimes\int_G [\lam(t^{-1})\check{T}_{ki}^*\check{T}_{kj}\lam(t)]\,dt.
\end{align*}
We observe that
\[
P_n=\int_G\lam(s)\otimes \lam(s)^{(n)}\,ds
\]
is evidently a self-adjoint projection.  Furthermore, the Schur orthogonality relations
tell us that on 
$\fH\otimes^2\fH^n=\ell^2\text{-}\bigoplus_{\pi',\pi\in\what{G}}\fH_{\pi'}\otimes^2\fH_\pi^n$
we have
\[
P_n=\bigoplus_{\pi\in\what{G}}\int_G \bar{\pi}(s)\otimes\pi(s)^{(n)}\,ds
\]
and hence this projection is non-zero only in each anti-diagonal component of \linebreak
$\vn(G)\bar{\otimes}\mat_n(\vn(G))=\ell^\infty\text{-}
\bigoplus_{\pi',\pi\in\what{G}\times\what{G}}\fB(\fH_{\pi'})
\bar{\otimes}\mat_n(\fB(\fH_\pi))$.  However, by Proposition \ref{prop:expectation} we see that
\[
\sum_{k=1}^nI\otimes\int_G [\lam(t^{-1})\check{T}_{ki}^*\check{T}_{kj}\lam(t)]\,dt
=\sum_{k=1}^n\bigoplus_{\pi\in\what{G}}\frac{1}{d_\pi}I\otimes
[\tr(\check{T}_{ki,\pi}^*\check{T}_{kj,\pi})I_\pi]
\]
Thus we see that
\begin{align*}
\norm{\Gam^{*(n)}[T_{ij}]}
&=\norm{\Gam^{*(n)}[T_{ij}]^*\Gam^{*(n)}[T_{ij}]}^{1/2} \\
&=\sup_{\pi\in\what{G}}\frac{1}{d_\pi^{1/2}}
\norm{\sum_{k=1}^n[\tr(\check{T}_{ki,\pi}^*\check{T}_{kj,\pi})]} ^{1/2} \\
&=\sup_{\pi\in\what{G}}\frac{1}{d_\pi^{1/2}}
\norm{\sum_{k=1}^n[\tr(T_{ki,\bar{\pi}}^*T_{kj,\bar{\pi}})]} ^{1/2}
\end{align*}
where we have appealed to Lemma \ref{lem:checkdual} in the last line.
According to \cite[(3.4.4)]{effrosr} we obtain completely isometric embedding
\[
\Gam^*(\vn(G))\subseteq\ell^\infty\text{-}\bigoplus_{\pi\in\what{G}}
\frac{1}{d_\pi^{1/2}}\smat^2_{d_\pi,c}
\]
Thus Lemma \ref{lem:opspace} and the duality of operator Hilbert spaces 
\cite[(3.4.4)]{effrosr} provide the desired result.  We observe, moreover, that 
our dual pairing (\ref{eq:trigdual}) is bilinear, which is why we need not concern ourselves 
with conjugate spaces.
\end{proof}

The operator space structure of Theorem \ref{theo:opsponfdel} provides alternate explanation
for curious results discovered in the some prior articles.

\begin{corollary}
{\bf (i)} {\rm \cite[Prop.\ 2.5]{forrestss}} We have $\fdelg\hat{\otimes}\fdelg\cong\fal_\Del(G\times G)$.

{\bf (ii)} {\rm \cite[Prop.\ 4.3]{forrestss2}}  The map $u\mapsto\check{u}$ is a complete isometry
on $\fdelg$.
\end{corollary}

\begin{proof}
Consider the identifications
\begin{align*}
\left(\ell^1\text{-}\bigoplus_{\pi\in\what{G}}d_\pi\smat_{d_\pi,r}^2\right)\hat{\otimes}
\left(\ell^1\text{-}\bigoplus_{\pi'\in\what{G}}d_{\pi'}\smat_{d_{\pi'},r}^2\right) 
&\cong \underset{\pi,\pi'\in\what{G}\times\what{G}}{\ell^1\text{-}\bigoplus}
d_\pi d_{\pi'}\smat_{d_\pi,r}^2
\hat{\otimes}\smat_{d_{\pi'},r}^2 \\
&\cong  \underset{\pi,\pi'\in\what{G}\times\what{G}}{\ell^1\text{-}\bigoplus}
d_\pi d_{\pi'}\smat_{d_\pi d_{\pi'},r}^2
\end{align*}
where we have appealed to the fact that the tensor product or row spaces is again a row space
(\cite[9.3.5]{effrosr}).  Since $\what{G\times G}\cong\what{G}\times
\what{G}$, (i) follows.  

The map $u\mapsto \check{u}:\trig_\pi\to\trig_{\bar{\pi}}$,
which is an isometry from $d_\pi\smat_{d_\pi}^2$ to 
$d_{\bar{\pi}}\smat_{d_{\bar{\pi}}}^2$ -- practically the transpose
-- is a complete isometry with row structure, thanks to \cite[3.4.2]{effrosr}.  
Thus by the structure of the direct product,
we get (ii).
\end{proof}

The next result is a bit of a surprise.  It shows that $\fal(G)\otimes^h\fal(G)$ behaves exactly
as does $\fal(G)\hat{\otimes}\fal(G)$ with respect to $\Gam$.

\begin{theorem}\label{theo:gamonhaag}
We have that $\Gam(\fal(G)\otimes^h\fal(G))=\fdelg$.  Moreover,
if $\fdelg$ is given the operator space structure in Theorem \ref{theo:opsponfdel}, above,
then $\Gam:\fal(G)\otimes^h\fal(G)\to\fdelg$ is a complete quotient map.
\end{theorem}

\begin{proof}
Let $\fal^h_\Del(G)=\Gam(\fal(G)\otimes^h\fal(G))$, and assign it the operator space
structure which makes $\Gam$ a complete quotient map.  The completley contractive
inclusion $\fal(G)\hat{\otimes}\fal(G)\hookrightarrow\fal(G)\otimes^h\fal(G)$ gives, via
the fact that $\fdelg$ is a complete quotient of $\fal(G)\hat{\otimes}\fal(G)$, a
completely contractive inclusion $\fdelg\hookrightarrow\fal^h_\Del(G)$.
Since $\trig(G)$ is dense in both subspaces, $\fdelg$ is dense in $\fal^h_\Del(G)$,
so there the adjoint $\fal^h_\Del(G)^*\hookrightarrow\fdelg^*$ is completely contractive and injective.  
We wish to see that this map is a complete isometry.  It suffices to appeal
to Lemma \ref{lem:opspace} and verify that on $\matr_n(\vn(G))$, we have
\begin{equation}\label{eq:dualineq}
\norm{[T_{ij}]}_{\matr_n(\fal^h_\Del(G)^*)}\leq\norm{[T_{ij}]}_{\matr_n(\fal_\Del(G)^*)}.
\end{equation}

We recall the duality relation (\ref{eq:fahtpdual}).
Hence for $[T_{ij}]$ in $\matr_n(\vn(G))$ we have by (\ref{eq:gamadjoint}) and the last
identification immediatley above
\begin{align*}
&\norm{[T_{ij}]}_{\matr_n(\fal^h_\Del(G)^*)}
=\norm{\Gam^{*(n)}[T_{ij}]}_{\fC\fB(\fB(\fH),\matr_n(\fB(\fH)))} \\
&\phantom{mm}
=\norm{A\mapsto \int_G [\lam(s)A\check{T}_{ij}\lam(s)]\,ds}_{\fC\fB(\fB(\fH),\matr_n(\fB(\fH)))} \\
&\phantom{mm}=\sup\left\{\left|\int_G \inprod{[\lam(s)A_{kl}\check{T}_{ij}\lam(s)]\xi}{\eta}\,ds\right|:
\begin{matrix} [A_{kl}]\in\mathrm{ball}(\matr_m(\fB(\fH))) \\
\xi,\eta\in\mathrm{ball}(\fH^{mn}),\, m\in\En \end{matrix}\right\}.
\end{align*}
We observe for operator matrix $[A_{kl}]$ and vectors $\xi,\eta$ as above that
\begin{align*}
&\left|\int_G \inprod{[\lam(s)A_{kl}\check{T}_{ij}\lam(s)]\xi}{\eta}\,ds\right| \\
&\phantom{mmm}\leq \int_G\norm{[A_{kl}\check{T}_{ij}\lam(s)]\xi}\norm{\lam(s^{-1})^{(nm)}\eta}\,ds \\
&\phantom{mmm}\leq \left( \int_G\norm{[A_{kl}\check{T}_{ij}\lam(s)]\xi}^2\right)^{1/2}
\left( \int_G\norm{\lam(s^{-1})^{(nm)}\eta}^2\,ds\right)^{1/2} \\
&\phantom{mmm}\leq  \left( \int_G
\inprod{[A_{kl}\check{T}_{ij}\lam(s)]^*[A_{kl}\check{T}_{ij}\lam(s)]\xi}{\xi}\right)^{1/2} \\
&\phantom{mmm}\leq 
\inprod{\int_G \sum_{k=1}^n[\lam(s^{-1})\check{T}_{ki}^*\check{T}_{kj}\lam(s)]\xi}{\xi}^{1/2} \\
&\phantom{mmm}\leq 
\norm {\int_G \sum_{k=1}^n[\lam(s^{-1})\check{T}_{ki}^*\check{T}_{kj}\lam(s)]}^{1/2}.
\end{align*}
However, calculations in the proof of Theorem \ref{theo:opsponfdel} reveal that the last
quantity is exactly
\[
\sup_{\pi\in\what{G}}\frac{1}{d_\pi^{1/2}}
\norm{\sum_{k=1}^n[\tr(T_{ki,\bar{\pi}}^*T_{kj,\bar{\pi}})]} ^{1/2}
=\norm{[T_{ij}]}_{\matr_n(\fal_\Del(G)^*)}.
\]
Thus (\ref{eq:dualineq}) is established.
\end{proof}

\subsection{Convolution}\label{sec:convolution}
We now wish to consider the map $\check{\Gam}$ on $\fal(G)\otimes^h\fal(G)$.
Consider $\check{\Gam}:\trig(G)\otimes\trig(G)\to\trig(G)$.  If
$u,v\in\trig(G)$ and $t\in G$ we have
\[
\langle \check{\Gam}(u\otimes v),\lam(t)\rangle=
\int_G u(s)v(s^{-1}t)\; ds=\int_G\langle u\otimes v,\lam(s)\otimes\lam(s^{-1})\lam(t)\rangle\,ds
\]
and hence we have that $\Gam^*(\lam(t))=\int_G\lam(s)\otimes\lam(s^{-1})\lam(t)\, ds$, in a 
weak* sense.  By weak*-density of $\spn\lam(G)$ in $\trig(G)'$ we conclude that
for $T$ in $\trig(G)'$ we have
\begin{equation}\label{eq:cgamadjoint}
\check{\Gam}^*(T)=\int_G\lam(s)\otimes\lam(s^{-1})T\, ds
\end{equation}
where the integral is understood in the weak* sense.

\begin{theorem}\label{theo:cgamonhaag}
We have $\check{\Gam}(\fal(G)\otimes^h\fal(G))=\fal(G)$.  Morover, 
$\check{\Gam}:\fal(G)\otimes^h\fal(G)\to\fal(G)$ is a complete quotient map.
\end{theorem}

\begin{proof}
Let $T\in\vn(G)$.  Then $\check{\Gam}^*(T)$, via the identification
(\ref{eq:fahtpdual}) and using the composition product of (\ref{eq:cbcomposition}),  factors as
\begin{gather*}
\check{\Gam}^*(T)=\int_G\lam(s)\otimes\lam(s^{-1})T\, ds
=(I\otimes T)\circ\int_G\lam(s)\otimes\lam(s^{-1})\, ds \\
\text{i.e.\ } \check{\Gam}^*(T)(A)=\left(\int_G \lam(s)A\lam(s^{-1})\,ds\right)T.
\end{gather*}
Hence if $[T_{ij}]\in\matr_n(\vn(G))$ then, since the map from Proposition \ref{prop:expectation}
is an expectation, hence completely contractive, we have
\begin{align*}
&\norm{\check{\Gam}^{*(n)}[T_{ij}]}_{\fC\fB(\fB(\fH),\matr_n(\fB(\fH)))} \\
&\phantom{mmm}\leq\norm{A\mapsto [AT_{ij}]}_{\fC\fB(\fB(\fH),\matr_n(\fB(\fH)))}
\norm{A\mapsto \int_G\lam(s)A\lam(s^{-1})\,ds}_{\fC\fB(\fB(\fH))} \\
&\phantom{mmm}\leq\norm{[T_{ij}]}_{\matr_n(\vn(G))}.
\end{align*}
Conversely, inspecting this operator at $A=I$ we observe
\[
\norm{\check{\Gam}^{*(n)}[T_{ij}]}_{\fC\fB(\fB(\fH),\matr_n(\fB(\fH)))}
\geq \norm{[T_{ij}]}_{\matr_n(\vn(G))}.
\]
Thus equality holds.  

Since $\check{\Gam}(\fal(G)\otimes^h\fal(G))$ contains
$\trig(G)$ which is dense in $\fal(G)$, and since $\check{\Gam}^*$ is a complete isometry, 
the desired results hold.
\end{proof}

We obtain an immediate corollary of Theorem  \ref{theo:cgamonhaag}, by
using the result \cite{blecher} 
(see \cite[5.2.1]{blecherl}) that any algebra and operator space
$\fA$ whose product completely boundedly 
factors through $\fA\otimes^h\fA$ is completely isomorphic
to an operator algebra, i.e., an algebra of operators acting on a Hilbert space.

\begin{corollary}
The convolution algebra $(\fal(G),\ast)=(\fal(G),\check{\Gam})$ is 
completely isomorphic to an operator algebra.
\end{corollary}

In fact, by \cite[5.2.8]{blecherl}, we see that the representation $\rho:(\falg,\ast)\to\fB(\fK)$
can be set to satisfy $\norm{\rho}_{\mathrm{cb}}\leq 2$ and $\norm{\rho^{-1}}_{\mathrm{cb}}\leq 1$
(the latter on $\rho(\falg)$).  Since $(\falg,\ast)$ is a Segal algebra in $\mathrm{L}^1(G)$, it does not
admit a bounded approximate identity (\cite[Theo.\ 1.2]{burnham}), in particular it does not admit
a contractive approximate identity.  Hence we do not have enough information
to establish if $\rho$ can be made a complete isometry (as would follow from the 
Blecher-Ruan-Sinclair theorem as stated in \cite[2.3.2]{blecherl}, for example).

This stands in mild contrast to  \cite[5.5.8]{blecherl}, where it is shown that
the `matrix' algebra $\smat^1_\infty$ is not an operator algebra.  Of course, the global
structure of $(\fal(G),\ast)$ is much different, since convolution on
$\trig_\pi\otimes\trig_\pi\cong d_\pi\smat^1_{d_\pi}\otimes d_\pi\smat^1_{d_\pi}$ is really
matrix multiplication times a scalar factor $\frac{1}{d_\pi}$ (see \cite[(27.20)]{hewittrII}); hence
we obtain a contraction $d_\pi\smat^1_{d_\pi}\otimes^hd_\pi\smat^1_{d_\pi}\to
d_\pi\smat^1_{d_\pi}$.

Let us close this section
with a remark on convolution applied to $\fal(G\times G)$.  The methods are very close 
to those of \cite[Thm.\ 4.1]{forrestss}, but we make the extra effort to gain the operator space structure.

\begin{proposition}
We have a completely isometric identification
\[
\check{\Gam}(\fal(G\times G))=\ell^1\text{-}\bigoplus_{\pi\in\what{G}}d_\pi^2\smat^1_{d_\pi}
\]
where we regard this space as a complete quotient space of $\fal(G\times G)$ by $\check{\Gam}$.
\end{proposition}

We remark that the space above was denoted $\fgamg$ in \cite{johnson}.  In
\cite[Def.\ 2.6]{lees} $\fgamg$ is regarded as a Beurling-Fourier algebra, 
and is given the same operator 
space structure, though in terms of a certain weighted dual pairing with $\vn(G)$,
which is a different perspective than the one taken here.

\begin{proof}
As above, we need only determine the norm of $\check{\Gam}^{*(n)}[T_{ij}]$ in \linebreak
$\matr_n(\vn(G)\bar{\otimes}\vn(G))\cong\matr_n(\vn(G\times G))$, for
$[T_{ij}]$ in $\matr_n(\vn(G))$.  Using (\ref{eq:cgamadjoint}) and factoring in the product
in $\mat_n(\vn(G)\bar{\otimes}\vn(G))$ we have
\begin{align*}
\check{\Gam}^{*(n)}[T_{ij}]&=\left(\int_G \lam(s)\otimes\lam(s^{-1})\, ds\right)^{(n)}[I\otimes T_{ij}] \\
&=\bigoplus_{\pi\in\what{G}}\left(\int_G\pi(s)\otimes\pi(s^{-1})\,ds\right)^{(n)}[I\otimes T_{ij,\pi}].  
\end{align*}
It is straightforwrd to calculate that each $\int_G\pi(s)\otimes\pi(s^{-1})\,ds=\frac{1}{d_\pi}U_\pi$
where $U_\pi$ is a unitary, in fact a permutation matrix.  Thus we obtain a completely
isometric embedding
\[
\check{\Gam}^*(\vn(G))
\subseteq\ell^\infty\text{-}\bigoplus_{\pi\in\what{G}}\frac{1}{d_\pi}\smat_{d_\pi}^\infty
\]
and the desired result follows.
\end{proof}

\section{Comparison of results}\label{sec:comparison}

We let $\fC(G)$ denote the space of continuous functions on $G$.  The
{\it Varopoulos algebra} is given by
\[
\mathrm{V}(G\times G)=\fC(G)\otimes^\gam\fC(G)=\fC(G)\hat{\otimes}\fC(G)
=\fC(G)\otimes^h\fC(G)
\]
where isomorphic equality of the spaces is provided by Grothendieck's inequality
\cite[14.5]{defantf}.  We shall take $\otimes^h$ to define our canonical 
norm on $\mathrm{V}(G\times G)$.
Since the map $u\mapsto\check{u}$ is a complete isometry
on $\fC(G)$, we have that $\check{\Gam}(\mathrm{V}(G\times G))=
\Gam(\mathrm{V}(G\times G))$ completely isometrically.  We recall that
$\falg\hat{\otimes}\falg=\fal(G\times G)$ completely isometrically, by virtue of
the facts that $(\falg\hat{\otimes}\falg)^*=\vn(G)\bar{\otimes}\vn(G)$,
as indicated in the proof of Theorem \ref{theo:opsponfdel}, and the latter
space is $\vn(G\times G)$; while
$\falg\otimes^\gam\falg=\fal(G\times G)$ isomorphically only when
$G$ has an abelian subgroup of finite index \cite{losert}, and isometrically only when
$G$ is abelian.  The fact that $u\mapsto\check{u}$
is an isometry on $\fal(G)$ means that $\Gam(\falg\otimes^\gam\falg)=
\check{\Gam}(\falg\otimes^\gam\falg)$.

\begin{tabular}{|c|c|c|l|}\hline
algebra &   image under $\Gam$ & image under $\check{\Gam}$ &  references \\ \hline\hline
$\mathrm{V}(G\times G)$ & $\falg$ & $\falg$ &  \cite{spronkt} \\
$\falg\otimes^h\falg$ & $\fdelg$ & $\falg$ & \S\ref{sec:twistedconvolution}, \S\ref{sec:convolution} \\
$\fal(G\times G)$ & $\fdelg$ &  $\fgamg$ & \cite{forrestss} (\S\ref{sec:twistedconvolution}, 
\S\ref{sec:convolution}) \\
$\falg\otimes^\gam\falg$ & $\fgamg$ & $\fgamg$ & \cite{johnson}  \\ \hline
\end{tabular}

\noindent Notice that in each of the first three rows, $\Gam$ and $\check{\Gam}$ can
be regarded as a complete quotient map, as shown in \S\ref{sec:twistedconvolution} and 
\S\ref{sec:convolution}, above.

Let us use Theorems  \ref{theo:gamonhaag} and \ref{theo:cgamonhaag} 
to observe some further connections
between $\fal(G\times G)$ and $\fal(G)\otimes^h\fal(G)$, and also between
$\mathrm{V}(G\times G)$ and $\fal(G)\otimes^h\fal(G)$.
This addresses a question asked in \cite[p.\ 21]{daws}.  
We use the same definitions as in \cite{spronkt,forrestss}.
The equivalence of (i), (ii) and (iii) below, is an immediate consequence of
 \cite[Theo.\ 1.4]{forrestss}.  The equivalence of (i'), (ii') and (iii')
is proved essentially as \cite[Thm.\ 3.1]{spronkt}; see 
\cite[Lem.\ 2.3]{forrestss2}.  We refer the reader to those sources for the proof.

\begin{proposition}\label{theo:synthesis}
Let $\theta,\check{\theta}:G\times G\to G$ be given by
$\theta(s,t)=st^{-1}$ and $\check{\theta}(s,t)=st$.
Also, let $E\subset G$ be closed.  Then the following are equivalent:

{\bf (i)} $E$ is a set of spectral synthesis for $\fdelg$;

{\bf (ii)} $\theta^{-1}(E)$ is a set of spectral synthesis for $\fal(G\times G)$;

{\bf (iii)} $\theta^{-1}(E)$ is a set of spectral synthesis for $\fal(G)\otimes^h\fal(G)$.

\noindent Also, the following are equivalent:

{\bf (i')} $E$ is a set of spectral synthesis for $\fal(G)$;

{\bf (ii')} $\check{\theta}^{-1}(E)$ is a set of spectral synthesis for $\mathrm{V}(G\times G)$;

{\bf (iii')} $\check{\theta}^{-1}(E)$ is a set of spectral synthesis for $\fal(G)\otimes^h\fal(G)$.
\end{proposition}

It is well-known, see for example \cite{herz}, that point sets are spectral for
$\fal(G)$.  Since $\check{\theta}^{-1}(\{e\})=\{(s,s^{-1}):s\in G\}$, we gain the following.

\begin{corollary}
The anti-diagonal $\check{\Delta}=\{(s,s^{-1}):s\in G\}$ is a set of 
spectral synthesis for $\fal(G)\otimes^h\fal(G)$.
\end{corollary}

 This stands in marked contrast to the case
for $\fal(G\times G)$:  $\check{\Delta}$ is a set of spectral synthesis for $\fal(G\times G)$
if and only if the connected component of the identity $G_e$ is abelian (\cite[Thm.\ 2.5]{forrestss2}).

\bigskip
{\sc Acknowledgement.}  The authors are grateful to the anonymous referee for making valuable
comments which improved the exposition of the paper.


\begin{thebibliography}{99}

\bibitem{blecher}
D.P.~Blecher.
\newblock A completely bounded characterization of operator algebras.
\newblock {\em Math. Ann.} 303 (1995), no. 2, 227--239.

\bibitem{blecherl}
D.P.~Blecher and C.~Le Merdy.
\newblock {\em Operator algebras and their modules -- an operator space approach.}
volume~30 of {\em London Math. Soc., New Series}.
\newblock  Clarendon Press, Oxford, New York, 2004.

\bibitem{blechers}
D.P.~Blecher and R.R.~Smith.
\newblock The dual of the Haagerup tensor product.
\newblock {\em J. London Math. Soc. (2)}, 45 (1992), no.\ 1, 126--144.

\bibitem{burnham}
J.T.~Burnham.
\newblock  Closed ideals in subalgebras of Banach algebras. I. 
\newblock {\em Proc. Amer. Math. Soc.} 32 (1972), 551--555.

\bibitem{daws}
M. Daws.
\newblock Multipliers and abstract harmonic analysis.
\newblock  Talk at {\em Canadian Abstract harmonic Analysis Symposium}, Saskatoon, 2010.

See {\tt http://www1.maths.leeds.ac.uk/$\sim$mdaws/talks/cahas\underline{\phantom{n}}hand.pdf}.

\bibitem{defantf}
A.~Defant and K.~Floret.
\newblock {\em Tensor norms and operator ideals}, volume~163 of {\em
North-Holland Mathematics Studies}. 
\newblock North-Holland, Amsterdam, 1993.

\bibitem{effrosr0}
E.G. Effros and Z.-J. Ruan.
\newblock On approximation properties for operator spaces.
\newblock {\em Internat. J. Math.} 1 (1990), no. 2, 163--187.

\bibitem{effrosr}
E.G. Effros and Z.-J. Ruan.
\newblock {\em Operator Spaces}, volume~23 of {\em London Math. Soc., New
  Series}.
\newblock Claredon Press, Oxford, 2000.

\bibitem{effrosr1}
E.G. Effros and Z.-J. Ruan.
\newblock Operator space tensor products and Hopf convolution algebras.
\newblock {\em J. Operator Theory}, 50 (2003), no. 1, 131--156.

\bibitem{eymard}
P.~Eymard.
\newblock L'alg\`{e}bre de {F}ourier d'un groupe localement compact.
\newblock {\em Bull. Soc. Math. France}, 92 (1964), 181--236.

\bibitem{forrest}
B.E.~Forrest.
\newblock Fourier analysis on coset spaces.
\newblock {\em Rocky Mountain J. Math.} 28 (1998), no. 1, 173--190.

\bibitem{forrestss}
B.E.~Forrest, E.~Samei and N.~Spronk.
\newblock Convolutions on compact groups and Fourier algebras of coset spaces.
\newblock {\em Studia Math.} 196 (2010), no. 3, 223--249.

\bibitem{forrestss2}
B.E.~Forrest, E.~Samei and N.~Spronk.
\newblock Weak amenability of Fourier algebras on compact groups.
\newblock {\em Indiana Univ. Math. J.} 58 (2009), no. 3, 1379--1393.

\bibitem{herz}
C.~Herz.
\newblock Harmonic synthesis for subgroups.
\newblock  {\em Ann. Inst. Fourier (Grenoble)} 23 (1973), no. 3, 91--123.

\bibitem{hewittrII}
E.~Hewitt and K.A.~Ross.
\newblock {\em Abstract harmonic analysis. Vol. II.}, volume~152 of {\em
Die Grundlehren der mathematischen Wissenschaften}.
\newblock Springer-Verlag, New York-Berlin, 1970.

\bibitem{johnson}
B.E.~Johnson.
\newblock Non-amenability of the Fourier algebra of a compact group.
\newblock {\em J. London Math. Soc. (2)}, 50 (1994), no. 2, 361--374.

\bibitem{lees}
H.H.~Lee and E. Samei.
\newblock Beurling-Fourier algebras, operator amenability and Arens regularity.
\newblock {\em J. Funct. Anal.}, 262 (2012), no. 1, 167--209.

\bibitem{losert}
V.~Losert.
\newblock On tensor products of Fourier algebras.
\newblock {\em Arch. Math. (Basel)}, 43 (1984), no. 4, 370--372.

\bibitem{smith}
R.R.~Smith.
\newblock Completely bounded module maps and the Haagerup tensor product.
\newblock {\em J. Funct. Anal.}, 102 (1991), no. 1, 156--175.

\bibitem{spronk}
N.~Spronk.
\newblock Measurable Schur multipliers and completely bounded multipliers of the Fourier algebras.
\newblock {\em Proc. London Math. Soc. (3)}, 89 (2004), no. 1, 161--192.

\bibitem{spronkt}
N.~Spronk and L.~Turowska.
\newblock Spectral synthesis and operator synthesis for compact groups.
\newblock {\em J. London Math. Soc. (2)}, 66 (2002), no. 2, 361--376.

\bibitem{tomiyama}
J.~Tomiyama.
\newblock  Tensor products of commutative Banach algebras.
\newblock  {\em T\^{o}hoku Math. J. (2)}, 12 (1960), 147--154.

\end{thebibliography}
\end{document}